\documentclass{amsart}
\usepackage{amscd,amsmath,amssymb}
\usepackage{pstricks}
\usepackage{latexsym,amsbsy,mathrsfs}
\usepackage{xypic}
\usepackage{pgf,tikz}
\usepackage{graphicx,color}
\usepackage{xcolor}
\usepackage[colorlinks=true,linkcolor=blue,citecolor=blue]{hyperref}

\newtheorem{thm}{Theorem}[section]
\newtheorem{lem}[thm]{Lemma}

\theoremstyle{definition}

\newtheorem{defn}[thm]{Definition}

\newtheorem{rem}[thm]{Remark}

\numberwithin{equation}{thm}

\begin{document}
\title[The (ET4) axiom for Extriangulated Categories]
{The (ET4) axiom for Extriangulated Categories}

\author{Xiaoxue Kong, Zengqiang Lin$^*$ and Minxiong Wang}
\address{ School of Mathematical sciences, Huaqiao University,
	Quanzhou\quad 362021,  China.}

\thanks{This work was supported by the Natural Science Foundation of Fujian Province (Grant No. 2020J01075)}
\thanks{$^*$ The corresponding author.}
\thanks{Email: 2100437637@qq.com; zqlin@hqu.edu.cn; mxw@hqu.edu.cn.}


\begin{abstract}
	Extriangulated categories were introduced by Nakaoka and Palu, which is a simultaneous generalization of exact categories and triangulated categories. The axiom (ET4) for extriangulated categories  is an analogue of the octahedron axiom (TR4) for triangulated categories. In this paper, we introduce  homotopy cartesian squares in pre-extriangulated categories to investigate  the axiom (ET4). We provide several equivalent statements of the axiom $(\mathrm{ET4})$ and find out conditions under which the axiom is self-dual.
\end{abstract}

\maketitle

\section{introduction}
Exact categories  and triangulated categories are two fundamental structures in mathematics. Nakaoka and Palu introduced the notion of extriangulated categories \cite{NP}, which is a simultaneous generalization of exact categories and triangulated categories. The class of extriangulated categories not only contains exact categories and extension-closed subcategories of triangulated categories, but also contains examples which are neither exact nor triangulated \cite{NP,ZZ,HZZ}. This new notion provides a convenient setup for writing down proofs which apply to both exact categories and triangulated categories.

Roughly speaking, a {\em pre-extriangluated category} is an additive category $\mathscr{C}$ equipped with a class of $\mathbb{E}$-triangles satisfying certain morphism axioms (ET3) and (ET3)$^{\text{op}}$, where $\mathbb{E}:\mathscr{C}^{\text{op}}\times\mathscr{C}\rightarrow \text{Ab}$ is an additive bifunctor and $\mathbb{E}$-triangles $A\stackrel{f}{\longrightarrow}B\stackrel{g}{\longrightarrow}C\stackrel{\delta}{\dashrightarrow}$ are given by an additive realization $\mathfrak{s}$ of $\mathbb{E}$. A pre-extriangulated category $(\mathscr{C},\mathbb{E},\mathfrak{s})$ is called {\em extriangulated} if it satisfies axioms (ET4) and (ET4)$^{\text{op}}$, which are similar to the octahedron axiom (TR4) for triangulated categories. To strengthen the connection with axiom (TR4), it is proved that four shifted octahedrons (ET4-1), (ET4-2), (ET4-3) and (ET4-4) hold in extriangulated categories \cite{NP}; see Section 2.2 for details.

In this paper, we further investigate the axiom (ET4). We don't know whether each shifted octahedron implying axiom (ET4) in pre-extriangulated categories. Our first aim is to discuss the equivalent statements of the axiom (ET4). Note that the axiom (TR4) is self-dual but the axiom (ET4) is not. Our second aim is to find out conditions under which the axiom (ET4) is self-dual.

Recall that homotopy cartesian squares in triangulated categories are the triangulated analogues of the pushout and pullback squares in abelian categories. It is well-known that the axiom (TR4) is equivalent to the homotopy cartesian axiom \cite{Nee,Kr}. In this paper we introduce the notion of homotopy cartesian squares and propose a new shifted octahedron (ET4-5) in extriangulated categories, which provides a useful method to achieve our two goals. We have the following main results.

\begin{thm}
	Let $(\mathscr{C},\mathbb{E},\mathfrak{s})$ be a pre-extriangulated category. Then $(\mathrm{ET4}$-$1)$, $(\mathrm{ET4}$-$2)$ and $(\mathrm{ET4}$-$5)$ are equivalent and self-dual.
\end{thm}

\begin{thm}
	Let	$(\mathscr{C},\mathbb{E},\mathfrak{s})$ be a pre-extriangulated category.
	Then the following statements are equivalent $:$
	
	$(a)$ $\mathscr{C}$ satisfies  $(\mathrm{ET4})$.
	
	$(b)$ $\mathscr{C}$ satisfies  $(\mathrm{ET4}$-$1)$ and \textup{(ET0)}.
	
	$(c)$ $\mathscr{C}$ satisfies  $(\mathrm{ET4}$-$2)$ and \textup{(ET0)}.
	
	$(d)$ $\mathscr{C}$ satisfies  $(\mathrm{ET4}$-$3)$ and \textup{(ET0)}.
	
	$(e)$ $\mathscr{C}$ satisfies  $(\mathrm{ET4}$-$5)$ and \textup{(ET0)}.
\end{thm}

Recall that the first morphism in an $\mathbb{E}$-triangle is called an {\em inflation}. We say $\mathscr{C}$  satisfies (ET0) if a composition of two inflations is also an inflation.

\begin{thm}
	Let $(\mathscr{C},\mathbb{E},\mathfrak{s})$ be a pre-extriangulated category. If $\mathscr{C}$ satisfies  $(\mathrm{ET0})$  and  $(\mathrm{ET0})^{\textup{op}}$, then  $(\mathrm{ET4})$, $(\mathrm{ET4}$-$1)$, $(\mathrm{ET4}$-$2)$, $(\mathrm{ET4}$-$3)$, $(\mathrm{ET4}$-$4)$ and $(\mathrm{ET4}$-$5)$ are equivalent and self-dual.
\end{thm}

The remaining part of this paper is organized as follows. In Section 2, we make some preliminaries on extriangulated categories. In Section 3, we introduce homotopy cartesian squares in pre-extriangulated categories and prove our main results.

\section{Preliminaries}

In this section, we first collect some definitions and facts on extriangulated categories, then provide several shifted octahedrons. The basic reference is \cite{NP}.

\subsection{Defintions and facts on extriangulated categories}

Let $\mathscr{C}$ be an additive category equipped with a biadditive functor  $\mathbb{E}:\mathscr{C}^{op} \times \mathscr{C}\rightarrow \mathrm{Ab}$, where Ab is the category of abelian groups. For any objects  $A , C\in \mathscr{C}$, an element $\delta \in \mathbb{E}(C,A)$ is called an $\mathbb{E}$-$\mathit{extension}$.
For any  $a\in\mathscr{C}(A,A')$ and  $c\in\mathscr{C}(C,C')$, we denote by $a_{*}\delta$ the $\mathbb{E}$-extension $\mathbb{E}(C,a)(\delta)\in\mathbb{E}(C,A')$ and by $c^{*}\delta$ the $\mathbb{E}$-extension $\mathbb{E}(c,A)(\delta)\in\mathbb{E}(C',A)$.
The zero element  $0\in \mathbb{E}(C,A) $ is called the $\mathit{split}$ $\mathbb{E}$-$\mathit{extension}$.
Let $\delta\in \mathbb{E}(C,A)$ and $\delta'\in \mathbb{E}(C',A')$. A {\em morphism} $(a,c):\delta\rightarrow \delta'$ of $\mathbb{E}$-extensions is a pair of morphisms $a\in \mathscr{C}(A,A')$ and $c\in \mathscr{C}(C,C')$  such that $a_{*}\delta=c^{*}\delta'.$
We denote by $\delta\oplus\delta'\in\mathbb{E}(C\oplus C', A\oplus A')$ the element corresponding to $(\delta,0,0,\delta')$ through the following isomorphism
\begin{center}
	$\mathbb{E}(C\oplus C',A\oplus  A')  \cong \mathbb{E}(C,A)\oplus\mathbb{E}(C,A') \oplus\mathbb{E}(C',A) \oplus\mathbb{E}(C',A')$.
\end{center}

Let $A,C\in \mathscr{C}$ be any pair of objects. Two sequences of morphisms
$A\stackrel{x}{\longrightarrow}B\stackrel{y}{\longrightarrow}C$ and $A\stackrel{x'}{\longrightarrow}B'\stackrel{y'}{\longrightarrow}C$
are said to be $\mathit{equivalent}$ if there exists an isomorphism  $b \in \mathscr{C}(B,B')$ such that the following diagram
$$\xymatrix{
	A\ar@{=}[d]\ar[r]^x &B\ar[d]^b_{\simeq}\ar[r]^y &C\ar@{=}[d]\\
	A\ar[r]^{x'} &B'\ar[r]^{y'} &C}	$$
is commutative. We denote the equivalence class of $A\stackrel{x}{\longrightarrow}B\stackrel{y}{\longrightarrow}C$ by $[A\stackrel{x}{\longrightarrow}B\stackrel{y}{\longrightarrow}C]$. For any $A,C\in\mathscr{C}$, we denote as
$$0=[A\stackrel{\left(\begin{smallmatrix}
	1\\0
	\end{smallmatrix}\right) }{\longrightarrow}A\oplus C\stackrel{\left(\begin{smallmatrix}
	0,1\\
	\end{smallmatrix}\right) }{\longrightarrow}C].$$
For any $[A\stackrel{x}{\longrightarrow}B\stackrel{y}{\longrightarrow}C]$ and $[A'\stackrel{x'}{\longrightarrow}B'\stackrel{y'}{\longrightarrow}C']$, we denote as
$$[A\oplus A'\stackrel{x\oplus x'}{\longrightarrow}B\oplus B'\stackrel{y\oplus y'}{\longrightarrow}C\oplus C'] =[A\stackrel{x}{\longrightarrow}B\stackrel{y}{\longrightarrow}C] \oplus [A'\stackrel{x'}{\longrightarrow}B'
\stackrel{y'}{\longrightarrow}C'].$$

\begin{defn}(\cite[Definitions 2.9]{NP})
	Let $ \mathfrak{s} $ be a correspondence which associates an equivalence class
	$\mathfrak{s}(\delta)=[A\stackrel{x}{\longrightarrow}B\stackrel{y}{\longrightarrow}C]$
	to any $\mathbb{E}$-extension $\delta \in \mathbb{E}(C,A)$.
	If $\mathfrak{s}(\delta) =[A\stackrel{x}{\longrightarrow}B\stackrel{y}{\longrightarrow}C]$, then we say the sequence  $A\stackrel{x}{\longrightarrow}B\stackrel{y}{\longrightarrow}C$ $\mathit{realizes}$ $\delta$.
	This $\mathfrak{s}$  is called a $\mathit{realization}$ of $\mathbb{E}$, if it satisfies the following condition:
	
	Let $ (a,c) :\delta \rightarrow \delta' $ be a morphism of $\mathbb{E}$-triangles,
	$\mathfrak{s}(\delta) =[A\stackrel{x}{\longrightarrow}B\stackrel{y}{\longrightarrow}C]$  and
	$\mathfrak{s}(\delta') =[A'\stackrel{x'}{\longrightarrow}B'\stackrel{y'}{\longrightarrow}C']$.
	Then  there exists a morphism $ b \in \mathscr{C}(B, B') $ which makes the following diagram commutative.
	$$\xymatrix{
		A \ar[r]^x \ar[d]^a & B\ar[r]^y \ar[d]^{b} & C \ar[d]^c&\\
		A'\ar[r]^{x'} & B' \ar[r]^{y'} & C'   }$$
	In the above situation, we say that the triplet $(a,b,c)$ $\mathit{realizes}$ $(a,c)$.
\end{defn}

\begin{defn}(\cite[Defintion 2.10]{NP})	
	A realization $\mathfrak{s}$ of $\mathbb{E}$ is called $\mathit{additive}$ if it satisfies the following conditions.
	
	$(1)$ For any objects $A,C\in\mathscr{C}$, the split $\mathbb{E}$-extension $0 \in \mathbb{E}(C,A)$ satisfies $\mathfrak{s}(0) = 0$.
	
	$(2)$ For any $\mathbb{E}$-extensions $\delta \in \mathbb{E} (C,A)$ and $\delta'\in \mathbb{E}(C',A')$, we have 	
	$\mathfrak{s}(\delta \oplus \delta') = \mathfrak{s}(\delta) \oplus \mathfrak{s}(\delta')$.
\end{defn}

Let  $\mathfrak{s}$ be an additive realization of $\mathbb{E}$.
A sequence $A\stackrel{x}{\longrightarrow}B\stackrel{y}{\longrightarrow}C$ is called a $\mathit{conflation}$ if realizes some $\mathbb{E}$-extension $\delta \in \mathbb{E}(C,A)$. In this case, $x$ is called an $\mathit{inflation}$, $y$ is called a $\mathit{deflation}$ and  $A\stackrel{x}{\longrightarrow}B\stackrel{y}{\longrightarrow}C\stackrel{\delta}{\dashrightarrow}$ is called an $\mathbb{E}$-$\mathit{triangle}$.

Let $A\stackrel{x}{\longrightarrow}B\stackrel{y}{\longrightarrow}C\stackrel{\delta}{\dashrightarrow}$
and $A'\stackrel{x'}{\longrightarrow}B'\stackrel{y'}{\longrightarrow}C'\stackrel{\delta'}{\dashrightarrow}$ be $\mathbb{E}$-triangles.
If $ (a,c) :\delta \rightarrow \delta' $ is a morphism of $\mathbb{E}$-triangles and  $(a,b,c)$ realizes $(a,c)$, then we write it as
$$\xymatrix{
	A \ar[d]_{a} \ar[r]^{x} & B \ar[d]_{b} \ar[r]^{y} & C\ar[d]_{c} \ar@{-->}[r]^{\delta}& \\
	A'\ar[r]^{x'} &B'\ar[r]^{y'} & C'\ar@{-->}[r]^{\delta'}&}$$
and call  $(a,b,c)$ a $\mathit{morphism}$ of $\mathbb{E}$-triangles.

\begin{defn}(\cite[Definition 2.12]{NP}) Let $\mathscr{C}$ be an additive category.
	A triplet $(\mathscr{C} ,\mathbb{E},\mathfrak{s})$ is called a $\mathit{pre}$-$\mathit{extriangulated }$ $\mathit{ category}$ if it satisfies the following conditions.
	
	$(\mathrm{ET1})$ $\mathbb{E}:\mathscr{C}^{\text{op}} \times \mathscr{C}\rightarrow \mathrm{Ab}$  is an additive bifunctor.
	
	$(\mathrm{ET2})$ $\mathfrak{s}$ is an additive realization of $\mathbb{E}$.
	
	$(\mathrm{ET3})$ Each  commutative diagram
	$$\xymatrix{
		A \ar[d]_{a} \ar[r]^{x} & B \ar[d]_{b} \ar[r]^{y} & C\ar@{-->}[r]^{\delta}& \\
		A'\ar[r]^{x'} &B'\ar[r]^{y'} & C'\ar@{-->}[r]^{\delta'}&}$$
	whose rows are $\mathbb{E}$-triangles can be completed to a morphism of $\mathbb{E}$-triangles.
	
	$(\mathrm{ET3})^{\mathrm{op}}$ Each  commutative diagram
	$$\xymatrix{
		A  \ar[r]^{x} & B \ar[d]_{b} \ar[r]^{y} & C \ar[d]_{c}\ar@{-->}[r]^{\delta}&  \\
		A'\ar[r]^{x'} &B'\ar[r]^{y'} & C'\ar@{-->}[r]^{\delta'}&}
	$$
	whose rows are $\mathbb{E}$-triangles can be completed to a morphism of $\mathbb{E}$-triangles.
	
	If the triplet $(\mathscr{C} ,\mathbb{E},\mathfrak{s})$ moreover satisfies the following axioms, then it is called an {\em extriangulated category}:
	
	(ET4) Let $A\stackrel{f}{\longrightarrow}B\stackrel{f'}{\longrightarrow}D\stackrel{\delta}{\dashrightarrow}$   and $B\stackrel{g}{\longrightarrow}C\stackrel{g'}{\longrightarrow}F\stackrel{\delta'}{\dashrightarrow}$ be $\mathbb{E}$-triangles.
	There exists  a commutative diagram
	$$\xymatrix{
		A \ar@{=}[d]\ar[r]^{f} &B\ar[d]^{g} \ar[r]^{f'} & D\ar@{-->}[d]^{d}\ar@{-->}[r]^{\delta}& \\
		A \ar@{-->}[r]^{h} & C\ar[d]^{g'} \ar@{-->}[r]^{h'} & E \ar@{-->}[d]^{e}\ar@{-->}[r]^{\delta''}& \\
		& F\ar@{-->}[d]^{\delta'}\ar@{=}[r] & F\ar@{-->}[d]^{f'_{*}\delta'}&\\
		& & &   }$$		
	such that the second row and the third column are $\mathbb{E}$-triangles, moreover,  $d^{*}\delta''=\delta$ and
	$f_{*}\delta''=e^{*}\delta'$.
	
	(ET4)$^{\text{op}}$ Let $D\stackrel{f'}{\longrightarrow}A\stackrel{f}{\longrightarrow}B\stackrel{\delta}{\dashrightarrow}$  and $F\stackrel{g'}{\longrightarrow}B\stackrel{g}{\longrightarrow}C\stackrel{\delta'}{\dashrightarrow}$  be $\mathbb{E}$-triangles. Then there exists  a commutative diagram
	$$\xymatrix{
		D \ar@{=}[d]\ar@{-->}[r]^{d} &E\ar@{-->}[d]^{h'} \ar@{-->}[r]^{e} & F\ar[d]^{g'}\ar@{-->}[r]^{g'^{*}\delta}& \\
		D\ar[r]^{f'} & A\ar@{-->}[d]^{h} \ar[r]^{f} & B\ar[d]^{g} \ar@{-->}[r]^{\delta}&\\
		& C\ar@{-->}[d]^{\delta''}\ar@{=}[r] & C \ar@{-->}[d]^{\delta'}& \\ &&&}$$	
	such that the first row and the second column are $\mathbb{E}$-triangles, moreover,
	$\delta'=e_{*}\delta''$ and  $d_{*}\delta=g^{*}\delta''$.
\end{defn}

\begin{lem} \textup{(}\cite[Proposition 3.3]{NP}\textup{)}\label{lem2.0}
	Let $(\mathscr{C},\mathbb{E},\mathfrak{s})$ be a pre-extriangulated category. Then for any $\mathbb{E}$-triangle $A\overset{x}{\longrightarrow}B\overset{y}{\longrightarrow}C\overset{\delta}{\dashrightarrow}$, the following sequences are exact:
	$$\mathscr{C}(C,-)\rightarrow\mathscr{C}(B,-)\rightarrow\mathscr{C}(A,-)\rightarrow\mathbb{E}(C,-)\rightarrow\mathbb{E}(B,-),$$
	$$\mathscr{C}(-,A)\rightarrow\mathscr{C}(-,B)\rightarrow\mathscr{C}(-,C)\rightarrow\mathbb{E}(-,A)\rightarrow\mathbb{E}(-,B).$$
\end{lem}

\begin{lem} \textup{(}\cite[Corollary 3.5]{NP}\textup{)}\label{lem 2.1}
	Let	$(\mathscr{C},\mathbb{E},\mathfrak{s})$ be a pre-extriangulated category. Assume that the following
	$$\xymatrix{
		A \ar[d]_{a} \ar[r]^{x} & B  \ar[r]^{y}\ar[d]_{b} & C \ar[d]_{c} \ar@{-->}[r]^{\delta} &  \\
		A'\ar[r]^{x'} & B' \ar[r]^{y'} & C' \ar@{-->}[r]^{\delta'} &  }
	$$
	is a morphism of $\mathbb{E}$-triangles. Then the following statements are equivalent.
	
	\textup{(1)} $a$ factors through $x$.
	
	\textup{(2)} $a_*\delta=c^*\delta'=0$.
	
	\textup{(3)} $c$ factors through $y'$. 	
\end{lem}

\begin{lem} \textup{(}\cite[Corollary 3.6]{NP}\textup{)} \label{lem2.09}
	Let	$(\mathscr{C},\mathbb{E},\mathfrak{s})$ be a pre-extriangulated category. Assume that $(a,b,c)$ is a morphism of $\mathbb{E}$-triangles.
	If  two  of $a$, $b$ and $c$ are isomorphisms, then so is the third.
\end{lem}

\begin{lem} \textup{(}\cite[Proposition 3.7]{NP}\textup{)} \label{lem 2.10}
	Let $(\mathscr{C},\mathbb{E},\mathfrak{s})$ be a pre-extriangulated category. Assume that $A\stackrel{x}{\longrightarrow}B\stackrel{y}{\longrightarrow}C\stackrel{\delta}{\dashrightarrow}$ is an $\mathbb{E}$-triangle.
	If $a\in \mathscr{C}(A,A')$ and $c\in\mathscr{C}(C',C)$ are isomorphisms, then
	$A'\stackrel{xa^{-1}}{\longrightarrow}B\stackrel{c^{-1}y}{\longrightarrow}C'\stackrel{a_*c^*\delta}{\dashrightarrow}$ is also an $\mathbb{E}$-triangle.
\end{lem}

\subsection{Shifted octahedrons}

Axiom ($\mathrm{ET4}$)  is an analogue of the octahedron axiom ($\mathrm{TR4}$) for triangulated categories. For later use, we collect the shifted octahedrons in extriangulated categories from \cite{NP,LN}.

\begin{lem} \textup{(}\cite[Proposition 3.15]{NP}\textup{)} \label{lem2.11}
	Let $(\mathscr{C},\mathbb{E},\mathfrak{s})$ be a pre-extriangulated category. If $\mathscr{C}$ satisfies $(\mathrm{ET4})$, then  $\mathscr{C}$ satisfies $(\mathrm{ET4}$-$1):$
	
	Let  	$A_{1}\stackrel{x_{1}}{\longrightarrow}B_{1}\stackrel{y_{1}}{\longrightarrow}C\stackrel{\delta_{1}}{\dashrightarrow}$ and
	$A_{2}\stackrel{x_{2}}{\longrightarrow}B_{2}\stackrel{y_{2}}{\longrightarrow}C\stackrel{\delta_{2}}{\dashrightarrow}$
	be  $\mathbb{E}$-triangles. Then there is a commutative diagram
	$$\xymatrix{
		&{A_2}\ar@{-->}[d]^{m_2}\ar@{=}[r]&{A_2}\ar[d]^{x_2}\\
		{A_1}\ar@{=}[d]\ar@{-->}[r]^{m_1}&M\ar@{-->}[r]^{e_1}\ar@{-->}[d]^{e_2}&{B_2}\ar[d]^{y_2}\ar@{-->}[r]^{y_{2}^{*}\delta_{1}}&\\
		{A_1}\ar[r]^{x_1}&{B_1}\ar@{-->}[d]^{y_{1}^{*}\delta_{2}}\ar[r]^{y_1}&C\ar@{-->}[r]^{\delta_{1}}\ar@{-->}[d]^{\delta_{2}}&\\
		&&&}$$
	whose second row and second column are $\mathbb{E}$-triangles such that $m_{1*}\delta_{1}+m_{2*}\delta_{2}=0$.
\end{lem}

\begin{lem}\textup{(}\cite[Proposition 1.20]{LN}\textup{)} \label{lem2.2}
	Let $(\mathscr{C},\mathbb{E},\mathfrak{s})$ be a pre-extriangulated category. If $\mathscr{C}$ satisfies $(\mathrm{ET4}$-$1)$, then  $\mathscr{C}$ satisfies  $(\mathrm{ET4}$-$2):$
	
	Let $A\stackrel{x}{\longrightarrow}B\stackrel{y}{\longrightarrow}C\stackrel{\delta}{\dashrightarrow}$ be an $\mathbb{E}$-triangle,  $f:A\rightarrow D$ be a morphism, and $D\stackrel{d}{\longrightarrow}E\stackrel{e}{\longrightarrow}C\stackrel{f_{*}\delta}{\dashrightarrow}$ be an $\mathbb{E}$-triangle realizing $f_{*}\delta$. Then there is a morphism $g:B\rightarrow E$ such that the following diagram
	$$\xymatrix{
		A \ar[d]_{f} \ar[r]^{x} & B  \ar[r]^{y}\ar@{-->}[d]_{g} & C \ar@{=}[d] \ar@{-->}[r]^{\delta} &  \\
		D\ar[r]^{d} &E \ar[r]^{e} & C \ar@{-->}[r]^{f_{*}\delta} &    }
	$$
	is a morphism of $\mathbb{E}$-triangles such that $ A\stackrel{\left(\begin{smallmatrix}
		-f\\ x
		\end{smallmatrix}\right)}{\longrightarrow}D\oplus B\stackrel{\left(\begin{smallmatrix}
		d, g
		\end{smallmatrix}\right)}{\longrightarrow}E\stackrel{
		e^{*}\delta}{\dashrightarrow} $ is an $\mathbb{E}$-triangle.
\end{lem}

\begin{lem} \textup{(}\cite[Lemma 3.14]{NP}\textup{)}\label{lem2.12}
	Let $(\mathscr{C},\mathbb{E},\mathfrak{s})$ be a pre-extriangulated category. If $\mathscr{C}$ satisfies $(\mathrm{ET4})$, then $\mathscr{C}$ satisfies $(\mathrm{ET4}$-$3):$
	
	Let
	$A\stackrel{f}{\longrightarrow}B\stackrel{f'}{\longrightarrow}D\stackrel{\delta_{f}}{\dashrightarrow}$,
	$B\stackrel{g}{\longrightarrow}C\stackrel{g'}{\longrightarrow}F\stackrel{\delta_{g}}{\dashrightarrow}$ and
	$A\stackrel{h}{\longrightarrow}C\stackrel{h_{0}}{\longrightarrow}E_{0}\stackrel{\delta_{h}}{\dashrightarrow}$
	be  $\mathbb{E}$-triangles satisfying $h=gf$. Then there are morphisms $d_{0}: D\rightarrow E_0$ and
	$e_{0}: E_0\rightarrow F$  such that the following
	$$\xymatrix{
		A\ar@{=}[d] \ar[r]^{f} &B\ar[r]^{f'}\ar[d]^{g}  & D\ar@{-->}[d]^{d_0}\ar@{-->}[r]^{\delta_{f}}&\\
		A \ar[r]^{h} & C\ar[r]^{h_0}\ar[d]^{g'}  & E_0 \ar@{-->}[d]^{e_0} \ar@{-->}[r]^{\delta_{h}}&\\
		&F\ar@{=}[r] \ar@{-->}[d]^{\delta_{g}}&F\ar@{-->}[d]^{f'_{*}(\delta_{g})}&\\
		&&&}$$
	is a commutative diagram whose third column is an $\mathbb{E}$-triangle, moreover, $d_0^*(\delta_h)=\delta_f$ and $f_*\delta_h=e_0^*(\delta_g)$.
\end{lem}

\begin{lem}\textup{(}\cite[Proposition 3.17]{NP}\textup{)} \label{lem2.3}
	Let $(\mathscr{C},\mathbb{E},\mathfrak{s})$ be an extriangulated category. Then $\mathscr{C}$ satisfies $(\mathrm{ET4}$-$4):$
	
	Let $D\stackrel{f}{\longrightarrow}A\stackrel{f'}{\longrightarrow}C\stackrel{\delta_{f}}{\dashrightarrow},
	A\stackrel{g}{\longrightarrow}B\stackrel{g'}{\longrightarrow}F\stackrel{\delta_{g}}{\dashrightarrow}$ and
	$E\stackrel{h}{\longrightarrow}B\stackrel{h'}{\longrightarrow}C\stackrel{\delta_{h}}{\dashrightarrow}$ be $\mathbb{E}$-triangles satisfying $ h'g=f'$. Then there is an $\mathbb{E}$-triangle $D\stackrel{d}{\longrightarrow}E\stackrel{e}{\longrightarrow}F\stackrel{\theta }{\dashrightarrow} $ such that  the following diagram
	$$\xymatrix{
		D \ar@{-->}[d]^{d} \ar[r]^{f} & A\ar[r]^{f'}\ar[d]^{g}  & C \ar@{=}[d]\ar@{-->}[r]^{\delta_{f}} &\\
		E \ar@{-->}[d]^{e} \ar[r]^{h} & B\ar[r]^{h'}\ar[d]^{g'} & C \ar@{-->}[r]^{\delta_{h}} &\\
		F\ar@{=}[r]\ar@{-->}[d]^{\theta} & F\ar@{-->}[d]^{\delta_{g}}	 &   \\
		&&&}$$	
	is commutative and satisfies the following three conditions.
	
	$(1)$ $d_{* }(\delta_{f})=\delta_{h}$,
	
	$(2)$ $f_{*}(\theta)=\delta_{g}$,
	
	$(3)$ $g'^{*}(\theta)+h'^{*} (\delta_{f})=0$.
\end{lem}

We remark that the proof of Lemma \ref{lem2.3} uses both (ET4) and (ET4)$^{\text{op}}$.

\section{Main results}
\paragraph{} In this section we introduce the definition of homotopy cartesian squares in pre-extriangulated categories and  provide some equivalent statements of axiom (ET4). 
\begin{defn}
	The following commutative square
	$$\xymatrix{
		A \ar[d]^{f}\ar[r]^{x} &B \ar[d]^{g} \\
		C \ar[r]^{y} & D }$$		
	in a pre-extriangulated category $(\mathscr{C},\mathbb{E},\mathfrak{s})$ is called  $\mathit{homotopy }$ $\mathit{ cartesian }$ if the sequence  $ A\stackrel{\left(\begin{smallmatrix}
		-f\\	x
		\end{smallmatrix}\right)}{\longrightarrow}C\oplus B\stackrel{\left(\begin{smallmatrix}
		y, g
		\end{smallmatrix}\right)}{\longrightarrow}D\stackrel{
		\partial}{\dashrightarrow} $ is an $\mathbb{E}$-triangle, where $\partial$ is called a $\mathit{differential}$.
\end{defn}

\begin{rem}
	We restate  $(\mathrm{ET4}$-$2)$ in the language of homotopy cartesian squares:
	
	Let $ A\stackrel{x}{\longrightarrow}B\stackrel{y}{\longrightarrow}C\stackrel{\delta}{\dashrightarrow} $ be an  $\mathbb{E}$-triangle, $f:A\rightarrow D $ be a morphism, and  $D\stackrel{d}{\longrightarrow}E\stackrel{e}{\longrightarrow}C\stackrel{f_{*}\delta}{\dashrightarrow}$ be an $\mathbb{E}$-triangle realizing $f_{*}\delta$. Then there is a morphism $g: B\rightarrow E$ which gives a morphism of $\mathbb{E}$-triangles
	$$\xymatrix{
		A \ar[d]_{f} \ar[r]^{x} & B  \ar[r]^{y}\ar@{-->}[d]_{g} & C \ar@{=}[d] \ar@{-->}[r]^{\delta} &  \\
		D\ar[r]^{d} &E \ar[r]^{e} & C \ar@{-->}[r]^{f_{*}\delta} &    }
	$$
	and moreover, the leftmost square is homotopy catesian where $e^{*}\delta$ is the differential.
\end{rem}

\begin{thm}\label{thm1}
	Let $(\mathscr{C},\mathbb{E},\mathfrak{s})$ be a pre-extriangulated category. Then $\mathscr{C}$ satisfies $(\mathrm{ET4}$-$2)$ if and only if  $\mathscr{C}$ satisfies $(\mathrm{ET4}$-$5):$
	
	Let $ A_1\stackrel{x_{1}}{\longrightarrow}B_1\stackrel{y_{1}}{\longrightarrow}C\stackrel{\delta}{\dashrightarrow}$  and
	$ A_2\stackrel{x_{2}}{\longrightarrow}B_2\stackrel{y_{2}}{\longrightarrow}C\stackrel{\delta'}{\dashrightarrow} $
	be  $\mathbb{E}$-triangles, and $ g:B_1\rightarrow B_2 $ be a morphism such that $ y_{2} g=y_{1}$. Then there is a morphism $f: A_1\rightarrow A_2$ which gives a morphism of $\mathbb{E}$-triangles
	$$\xymatrix{
		A_1 \ar[r]^{x_1} \ar@{-->}[d]_{f}  & B_1  \ar[r]^{y_1}\ar[d]_{g} & C \ar@{=}[d] \ar@{-->}[r]^{\delta} &  \\
		A_2 \ar[r]^{x_2} &B_2 \ar[r]^{y_2} & C \ar@{-->}[r]^{\delta'} &    }
	$$
	and moreover, the leftmost square is homotopy cartesian where $y_2^{*}\delta$ is the differential.
\end{thm}

\begin{proof}
	($ \mathrm{ET4} $-$ 2 )  \Rightarrow  (\mathrm{ET4} $-$ 5 $).
	By $ (\mathrm{ET3})^{\textup{op}} $, there exists a morphism  $ f_1: A_{1}\rightarrow A_{2} $ which gives a morphism of $\mathbb{E}$-triangles
	$$\xymatrix{
		A_1 \ar@{-->}[d]_{f_1} \ar[r]^{x_1} & B_1  \ar[r]^{y_1}\ar[d]_{g} & C \ar@{=}[d] \ar@{-->}[r]^{\delta} &  \\
		A_2\ar[r]^{x_2} &B_2 \ar[r]^{y_2} & C \ar@{-->}[r]^{\delta'} & . }
	$$
	By ($\mathrm{ET4}$-$2$), there exists a morphism  $ g_1:B_{1}\rightarrow B_2 $ which gives a morphism of $\mathbb{E}$-triangles
	$$\xymatrix{
		A_1 \ar[d]_{f_1} \ar[r]^{x_1} & B_1  \ar[r]^{y_1}\ar@{-->}[d]_{g_1} & C \ar@{=}[d] \ar@{-->}[r]^{\delta} &  \\
		A_2\ar[r]^{x_2} &B_2 \ar[r]^{y_2} & C \ar@{-->}[r]^{\delta'} &  }
	$$
	and moreover, the leftmost square is a homotopy cartesian  where $y_2^{*}\delta$ is the differential. In this case, we have an $\mathbb{E}$-triangle $ A_1\stackrel{\left(\begin{smallmatrix}
		-f_1\\	x_1
		\end{smallmatrix}\right)}{\longrightarrow} A_2\oplus B_1 \stackrel{\left(\begin{smallmatrix}
		x_2 , g_1
		\end{smallmatrix}\right)}{\longrightarrow}B_2\stackrel{
		y_2^{*}\delta}{\dashrightarrow} $. Since $y_2(g-g_1)=0$, by Lemma \ref{lem2.0}
	there is a morphism $m: B_1\rightarrow A_2$ such that $g-g_1=x_2 m$. Set $ f = f_1+mx_1 $, then
	$x_2 f=x_2f_1+x_2mx_1=g_1x_1+(g-g_1)x_1=gx_1$. The following commutative diagram
	$$\xymatrix{
		A_1\ar[r]^{\left(\begin{smallmatrix}
			-f_1\\	x_1
			\end{smallmatrix}\right)}\ar@{=}[d] & A_2\oplus B_1\ar[r]^{\left(\begin{smallmatrix}
			x_2 , g_1
			\end{smallmatrix}\right)}\ar[d]^{\left(\begin{smallmatrix}
			1 & -m \\ 0 & 1
			\end{smallmatrix}\right)} & B_2\ar@{=}[d]\\
		A_1\ar[r]^{\left(\begin{smallmatrix}
			-f\\	x_1
			\end{smallmatrix}\right)} & A_2\oplus B_1\ar[r]^{\left(\begin{smallmatrix}
			x_2 , g
			\end{smallmatrix}\right)} & B_2\\
	}$$
	implies that $ A_1\stackrel{\left(\begin{smallmatrix}
		-f\\	x_1
		\end{smallmatrix}\right) }{\longrightarrow} A_2 \oplus B_1  \stackrel{\left(\begin{smallmatrix}
		x_2, g
		\end{smallmatrix}\right) }{\longrightarrow}B_2\stackrel{
		y_2^{*}\delta}{\dashrightarrow} $ is an $\mathbb{E}$-triangle.
	Note that $(mx_1)_*\delta=0$ by Lemma \ref{lem 2.1}. Thus $f_*\delta=(f_1)_*\delta+(mx_1)_*\delta=(f_1)_*\delta=\delta'$ and (ET4-5) holds.
	
	($\mathrm{ET4}$-$5)  \Rightarrow  (\mathrm{ET4}$-$2$).
	Since $(f,\text{id}_C): \delta\rightarrow f_*\delta$ is a morphism of $\mathbb{E}$-extensions, there
	exists a morphism $g':B\rightarrow E$ which makes the diagram below commutative.
	$$\xymatrix{
		A \ar[d]_{f} \ar[r]^{x} & B  \ar[r]^{y}\ar@{-->}[d]_{g'} & C \ar@{=}[d] \ar@{-->}[r]^{\delta} &  \\
		D\ar[r]^{d} &E \ar[r]^{e} & C \ar@{-->}[r]^{f_{*}\delta} &    }
	$$
	By ($\mathrm{ET4}$-$5$), there  exists a morphism $f':A\rightarrow D$ which gives a morphism of $\mathbb{E}$-triangles
	$$\xymatrix{
		A \ar@{-->}[d]_{f'} \ar[r]^{x} & B  \ar[r]^{y}\ar[d]_{g'} & C \ar@{=}[d] \ar@{-->}[r]^{\delta} &  \\
		D\ar[r]^{d} &E \ar[r]^{e} & C \ar@{-->}[r]^{f_{*}\delta} &    }
	$$
	and moreover, the leftmost square is homotopy cartesian where $e^*\delta$ is a differential.
	Since $f'_*\delta=f_*\delta$, that is, $(f-f')_*\delta=0$,
	there exists a morphism $m: B\rightarrow D$ such that $f-f'=mx$ by Lemma \ref{lem 2.1}.
	Set $ g = g'+ dm $, then
	$gx=g'x+dmx=df'+d(f-f')=df$ and $eg=e(g'+dm)=y+edm=y$.
	The following commutative diagram
	$$\xymatrix{
		A\ar[r]^{\left(\begin{smallmatrix}
			-f'\\	x
			\end{smallmatrix}\right)}\ar@{=}[d] & D\oplus B\ar[r]^{\left(\begin{smallmatrix}
			d , g'
			\end{smallmatrix}\right)}\ar[d]^{\left(\begin{smallmatrix}
			1 & -m \\ 0 & 1
			\end{smallmatrix}\right)} & E\ar@{=}[d]\\
		A\ar[r]^{\left(\begin{smallmatrix}
			-f\\	x
			\end{smallmatrix}\right)} & D\oplus B\ar[r]^{\left(\begin{smallmatrix}
			d , g
			\end{smallmatrix}\right)} & E\\
	}$$
	implies that $ A\stackrel{ \left(\begin{smallmatrix}
		-f\\	x
		\end{smallmatrix}\right)}{\longrightarrow} D \oplus B  \stackrel{\left(\begin{smallmatrix}
		d, g
		\end{smallmatrix}\right)}{\longrightarrow}E\stackrel{
		e^{*}\delta}{\dashrightarrow} $ is an $\mathbb{E}$-triangle. Therefore, (ET4-2) holds.
\end{proof}

\begin{thm} \label{thm2}
	Let $(\mathscr{C},\mathbb{E},\mathfrak{s})$ be a pre-extriangulated category. Then $(\mathrm{ET4}$-$1)$, $(\mathrm{ET4}$-$2)$ and $(\mathrm{ET4}$-$5)$ are equivalent and self-dual.
\end{thm}

\begin{proof} Since (ET4-2) is equivalent to (ET4-5) by Theorem \ref{thm1}, and (ET4-1) implying (ET4-2) follows from Lemma \ref{lem2.2}, by duality it remains to show that (ET4-2) implying (ET4-1)$^{\text{op}}$.
	
	Suppose there is a diagram
	$$\xymatrix{
		A\ar[d]^{f} \ar[r]^{x} & B\ar[r]^{y}  & C\ar@{-->}[r]^{\delta_1} &  \\
		D\ar[d]^{g}&&&\\
		E\ar@{-->}[d]^{\delta_2} &&&\\
		&&&}$$ whose row and column are $\mathbb{E}$-triangles.
	By ($\mathrm{ET4}$-$2$), there exists a morphism $m_2:B\rightarrow M$  such that the following diagram
	$$\xymatrix{
		A \ar[d]_{f} \ar[r]^{x} & B  \ar[r]^{y}\ar@{-->}[d]_{m_2} & C \ar@{=}[d] \ar@{-->}[r]^{\delta_1} &  \\
		D\ar[r]^{m_1} &M \ar[r]^{e_1} & C \ar@{-->}[r]^{f_{*}\delta_1} &  }	$$
	is commutative and
	$ A\stackrel{\left(\begin{smallmatrix}
		-f\\x
		\end{smallmatrix}\right)}{\longrightarrow} D\oplus B \stackrel{\left(\begin{smallmatrix}
		m_1, m_2
		\end{smallmatrix}\right)}{\longrightarrow}M\stackrel{
		e_1^{*}\delta_1}{\dashrightarrow}$ is an $\mathbb{E}$-triangle.
	By ($\mathrm{ET4}$-$2$) again, there exists a morphism $m'_1: D\rightarrow M'$ such that the following diagram
	$$\xymatrix{
		A \ar[d]_{x} \ar[r]^{f} & D  \ar[r]^{g}\ar@{-->}[d]_{m_1'} & E \ar@{=}[d] \ar@{-->}[r]^{\delta_2} &  \\
		B\ar[r]^{m_2'} &M' \ar[r]^{e_2'} & E \ar@{-->}[r]^{x_{*}\delta_2} &    }
	$$
	commutes and
	$ A\stackrel{\left(\begin{smallmatrix}
		-x\\f
		\end{smallmatrix}\right)}{\longrightarrow} B\oplus D \stackrel{\left(\begin{smallmatrix}
		m_2', m_1'
		\end{smallmatrix}\right)}{\longrightarrow}M'\stackrel{
		e_2'^{*}\delta_2}{\dashrightarrow}$ is an $\mathbb{E}$-triangle.
	By (ET3) and Lemma \ref{lem2.09}, there is an isomorphism $c:M'\rightarrow M$ such that the following diagram
	$$\xymatrix{
		A\ar@{=}[d] \ar[r]^{\left(\begin{smallmatrix}
			-f\\x
			\end{smallmatrix}\right)} & D\oplus B\ar[d]^{\left(
			\begin{smallmatrix}
			0 & -1 \\
			-1 & 0
			\end{smallmatrix}
			\right)
		} \ar[r]^{\left(\begin{smallmatrix}
			m_1', m_2'
			\end{smallmatrix}\right)} & M'\ar@{-->}[d]^{c} \ar@{-->}[r]^{
			e_2'^{*}\delta_2}& \\
		A\ar[r]^{\left(\begin{smallmatrix}
			-x\\f
			\end{smallmatrix}\right)} & B\oplus D  \ar[r]^{\left(\begin{smallmatrix}
			m_2, m_1
			\end{smallmatrix}\right)} & M \ar@{-->}[r]^{e_1^{*}\delta_1}&}
	$$ is commutative and $c^*(e_1^{*}\delta_1)=e_2'^*\delta_2$.
	Set $e_2=-e_2'c^{-1}$, then $e_1^*\delta_{1} + e_2^*\delta_{2} =0$ and we obtain a commutative diagram
	$$\xymatrix{B\ar[r]^{m_2}\ar@{=}[d]& M\ar[r]^{e_2}\ar[d]^{-c^{-1}}_{\simeq}&E\ar@{=}[d]\\
		B\ar[r]^{m_2'} & M'\ar[r]^{e_2'} & E. \\
	} $$
	Since $e_2m_1=-e_2'c^{-1}m_1=e_2'm_1'=g$, we have the following commutative diagram
	$$\xymatrix{
		A\ar[d]^{f} \ar[r]^{x} & B\ar[r]^{y}\ar[d]^{m_2}  & C\ar@{=}[d]\ar@{-->}[r]^{\delta_1} &  \\
		D\ar[d]^{g}\ar[r]^{m_1}& M\ar[r]^{e_1}\ar[d]^{e_2}&C\ar@{-->}[r]^{f_*\delta_1}&\\
		E\ar@{-->}[d]^{\delta_2}\ar@{=}[r] &E\ar@{-->}[d]^{x_*\delta_2}&&\\
		&&&}$$
	whose second row and second column are $\mathbb{E}$-triangles. Therefore, (ET4-1)$^{\text{op}}$ holds.
\end{proof}

Let	$(\mathscr{C},\mathbb{E},\mathfrak{s})$ be a pre-extriangulated category. Consider the following conditions:

(ET0) Let $f: A\rightarrow B$ and $g:B\rightarrow C$ be morphisms in $\mathscr{C}$. If both $f$ and $g$ are inflations, then so is  $gf$.

(ET0)$^{\textup{op}}$ Let $f: A\rightarrow B$ and $g:B\rightarrow C$ be morphisms in $\mathscr{C}$. If both $f$ and $g$  are deflations, then so is  $gf$.

\begin{thm}\label{thm3}
	Let	$(\mathscr{C},\mathbb{E},\mathfrak{s})$ be a pre-extriangulated category.
	Then the following statements are equivalent $:$
	
	$(a)$ $\mathscr{C}$ satisfies  $(\mathrm{ET4})$.
	
	$(b)$ $\mathscr{C}$ satisfies  $(\mathrm{ET4}$-$1)$ and \textup{(ET0)}.
	
	$(c)$ $\mathscr{C}$ satisfies  $(\mathrm{ET4}$-$2)$ and \textup{(ET0)}.
	
	$(d)$ $\mathscr{C}$ satisfies  $(\mathrm{ET4}$-$3)$ and \textup{(ET0)}.
	
	$(e)$ $\mathscr{C}$ satisfies  $(\mathrm{ET4}$-$5)$ and \textup{(ET0)}.
\end{thm}

\begin{proof}
	By Theorem \ref{thm2}, we have (b)$\Leftrightarrow$ (c) $\Leftrightarrow$ (e). By  Lemma \ref{lem2.11}, we have (a) implying (b). Thus  we only need to show that (e) implying (d) and (d) implying (a).
	
	$(e) \Rightarrow (d)$. We will show that (ET4-5) implying (ET4-3). The proof is an analogue of \cite[3.5]{Hu}; see also \cite[Lemma 4.2]{HLN}.
	Assume that $A\stackrel{f}{\longrightarrow}B\stackrel{f'}{\longrightarrow}D\stackrel{\delta}{\dashrightarrow}$, $B\stackrel{g}{\longrightarrow}C\stackrel{g'}{\longrightarrow}F\stackrel{\delta'}{\dashrightarrow}$ and $A\stackrel{h}{\longrightarrow}C\stackrel{h'}{\longrightarrow}E\stackrel{\delta''}{\dashrightarrow}$ be  $\mathbb{E}$-triangles satisfying $h=g f$. Note that (ET4-5) is equivalent to (ET4-5)$^{\text{op}}$ by Theorem \ref{thm2}. Thus there is a morphism $d: D\rightarrow E$ such that the following diagram
	$$\xymatrix{
		A\ar@{=}[d] \ar[r]^{f} & B  \ar[r]^{f'}\ar[d]_{g} & D  \ar@{-->}[d]_{d} \ar@{-->}[r]^{\delta} &  \\
		A\ar[r]^{h} &C \ar[r]^{h'} &E\ar@{-->}[r]^{\delta''} &    }
	$$
	is a morphism of $\mathbb{E}$-triangles and the middle square is homotopy cartesian  where $f_*\delta''$ is the differential. Thus $d^*\delta''=\delta$ and $ B\stackrel{\left(\begin{smallmatrix}
		-f'\\	g
		\end{smallmatrix}\right)}{\longrightarrow} D \oplus C \stackrel{\left(\begin{smallmatrix}
		d , h'
		\end{smallmatrix}\right)}{\longrightarrow} E \stackrel{
		f_{*}\delta''}{\dashrightarrow} $ is an $\mathbb{E}$-triangle. By (ET4-5)$^{\text{op}}$ again
	there is a morphism $e:E\rightarrow F$ such that the following diagram
	$$\xymatrix{
		B\ar@{=}[d] \ar[r]^{\left(\begin{smallmatrix}
			-f'\\	g
			\end{smallmatrix}\right)} &  D \oplus C  \ar[r]^{\left(\begin{smallmatrix}
			d , h'
			\end{smallmatrix}\right)}\ar[d]^{\left(\begin{smallmatrix}
			0 ,	1
			\end{smallmatrix}\right)} & E \ar@{-->}[d]^{e}  \ar@{-->}[r]^{f_{*}\delta''} &  \\
		B\ar[r]^{g} &C \ar[r]^{g'} &F\ar@{-->}[r]^{\delta'} &    }
	$$
	is a morphism of $\mathbb{E}$-triangles and the middle square is homotopy cartesian where $\left(\begin{smallmatrix}
	-f' \\ g
	\end{smallmatrix}\right)_{*}\delta'$ is the differential. Thus $e^*\delta'=f_*\delta''$ and
	$$\xymatrix{
		D \oplus C \ar[r]^{\left(\begin{smallmatrix}
			-d & -h'\\	0 & 1
			\end{smallmatrix}\right)} &  E \oplus C  \ar[r]^{\left(\begin{smallmatrix}
			e , g'
			\end{smallmatrix}\right)} & F   \ar@{-->}[r]^{\left(\begin{smallmatrix}
			-f' \\ g
			\end{smallmatrix}\right)_{*}\delta'} &}
	$$
	is an $\mathbb{E}$-triangle.
	Note that we have the following  morphisms of $\mathbb{E}$-triangles.
	$$\xymatrix{
		D \oplus C \ar@{=}[d] \ar[r]^{\left(\begin{smallmatrix}
			-d & -h'\\	0 & 1
			\end{smallmatrix}\right)} &  E \oplus C \ar[d]^{\left(\begin{smallmatrix}
			1 & h' \\  0 & 1
			\end{smallmatrix}\right)}  \ar[r]^{\begin{pmatrix}
			e , g'
			\end{pmatrix}}\ar[d]& F \ar@{=}[d] \ar@{-->}[r]^{\left(\begin{smallmatrix}
			-f' \\ g
			\end{smallmatrix}\right)_{*}\delta'} &  \\
		D \oplus C  \ar[r]^{\left(\begin{smallmatrix}
			-d & 0\\	0 & 1
			\end{smallmatrix}\right)} \ar[d]^{\left(\begin{smallmatrix}
			-1 , 0
			\end{smallmatrix}\right)} & E \oplus C  \ar[d]^{\left(\begin{smallmatrix}
			1 , 0
			\end{smallmatrix}\right)} \ar[r]^{\left(\begin{smallmatrix}
			e , 0
			\end{smallmatrix}\right)} &F\ar@{=}[d] \ar@{-->}[r]^{\left(\begin{smallmatrix}
			-f' \\ g
			\end{smallmatrix}\right)_{*}\delta'} & \\
		B\ar[r]_{d} & E \ar[r]_{e} &F\ar@{-->}[r]_{\left(\begin{smallmatrix}
			-1 , 0
			\end{smallmatrix}\right) \left(\begin{smallmatrix}
			-f' \\ g
			\end{smallmatrix}\right)_{*}\delta'} &   }
	$$
	Since $ \begin{pmatrix}
	-1 , 0
	\end{pmatrix}_{*}
	\begin{pmatrix}
	-f' \\ g
	\end{pmatrix}_{*} \delta' =f'_* \delta' 	
	$,  we have the following commutative diagram
	$$\xymatrix{
		A\ar@{=}[d] \ar[r]^{f} &B\ar[r]^{f'}\ar[d]^{g}  & D\ar@{-->}[d]^{d}\ar@{-->}[r]^{\delta}&\\
		A \ar[r]^{h} & C\ar[r]^{h'}\ar[d]^{g'}  & E \ar@{-->}[d]^{e} \ar@{-->}[r]^{\delta''}&\\
		&F\ar@{=}[r] \ar@{-->}[d]^{\delta'}&F\ar@{-->}[d]^{f'_{*}\delta'}&\\
		&&&}$$
	such that the third column is an $\mathbb{E}$-triangle, $d^*\delta''=\delta$ and $e^*\delta'=f_*\delta''$.
	
	$(d)\Rightarrow (a)$. Let
	$A\stackrel{f}{\longrightarrow}B\stackrel{f'}{\longrightarrow}D\stackrel{\delta_{f}}{\dashrightarrow}$ and
	$B\stackrel{g}{\longrightarrow}C\stackrel{g'}{\longrightarrow}F\stackrel{\delta_{g}}{\dashrightarrow}$
	be two $\mathbb{E}$-triangles. By (ET0), $h=gf:A \rightarrow C$ is an inflation. Assume that	$A\stackrel{h}{\longrightarrow}C\stackrel{h_{0}}{\longrightarrow}E_{0}\stackrel{\delta_{h}}{\dashrightarrow}$ is an $\mathbb{E}$-triangle.
	By {\rm (ET4-3)},
	there are morphisms $d_0: D\rightarrow E_0$ and $e_0: E_0\rightarrow F$ such that  the following diagram
	$$\xymatrix{
		A\ar@{=}[d] \ar[r]^{f} &B\ar[r]^{f'}\ar[d]^{g}  & D\ar@{-->}[d]^{d_0}\ar@{-->}[r]^{\delta_{f}}&\\
		A \ar[r]^{h} & C\ar[r]^{h_0}\ar[d]^{g'}  & E_0 \ar@{-->}[d]^{e_0} \ar@{-->}[r]^{\delta_{h}}&\\
		&F\ar@{=}[r] \ar@{-->}[d]^{\delta_{g}}&F\ar@{-->}[d]^{f'_{*}(\delta_{g})}&\\
		&&&}
	$$
	is commutative 	and the third column is an $\mathbb{E}$-triangle, $d_{0}^{*}(\delta_{h})=\delta_{f}$ and
	$f_{*}\delta_{h}=e_{0}^{*}(\delta_{g})$. Therefore, (ET4) holds.
\end{proof}

\begin{lem}\label{lem3.1}
	Let	$(\mathscr{C},\mathbb{E},\mathfrak{s})$ be a pre-extriangulated category. If $\mathscr{C}$ satisfies  $(\mathrm{ET4}$-$4)$, then $\mathscr{C}$ satisfies  $(\mathrm{ET4}$-$5)$.
\end{lem}

\begin{proof}
	Suppose we have the following commutative diagram
	$$\xymatrix{
		A_{1} \ar[r]^{x_{1}} & B_{1} \ar[r]^{y_{1}}\ar[d]^{g} & C\ar@{=}[d]  \ar@{-->}[r]^{\delta} &  \\
		A_{2} \ar[r]^{x_{2}} & B_{2} \ar[r]^{y_{2}} &C\ar@{-->}[r]^{\delta^{'}} &    }
	$$
	whose rows are $\mathbb{E}$-triangles.
	Consider the following commutative diagram
	$$\xymatrix{
		A_2\oplus B_{1} \ar[r]^{\left(\begin{smallmatrix}
			x_2 & 0 \\ 0 & 1
			\end{smallmatrix}\right)}\ar@{=}[d] & B_{2} \oplus B_{1} \ar[r]^{\left(\begin{smallmatrix}
			y_2 , 0
			\end{smallmatrix}\right)}\ar[d]^{\left(\begin{smallmatrix}
			1& g\\  0&1
			\end{smallmatrix}\right)} & C\ar[d]^{-1}\ar@{-->}[r]^{\delta' \oplus 0} & \\
		A_2\oplus B_{1} \ar[r]^{\left(\begin{smallmatrix}
			x_2 & g\\ 0 & 1
			\end{smallmatrix}\right)} & B_{2} \oplus B_{1} \ar[r]^{\left(\begin{smallmatrix}
			-y_2, y_1
			\end{smallmatrix}\right)} & C\ar@{-->}[r]^{-(\delta' \oplus 0)}& .   }
	$$
	Since the first row is an $\mathbb{E}$-triangle, the second row is
	also an $\mathbb{E}$-triangle by Lemma \ref{lem 2.10}.
	
	By $(\mathrm{ET4}$-$4)$, we have the following commutative diagram
	$$\xymatrix{
		A_{1}\ar@{-->}[d]_{\left(\begin{smallmatrix}
			-f\\ x_{1}
			\end{smallmatrix}\right)} \ar[r]^{x_{1 }} & B_1 \ar[r]^{y_1}\ar[d]_{\left(\begin{smallmatrix}
			0\\1
			\end{smallmatrix}\right)}  & C\ar@{=}[d]\ar@{-->}[r]^{\delta} &\\
		A_{2} \oplus B_{1}\ar@{-->}[d]_{\left(\begin{smallmatrix}
			x_2, g
			\end{smallmatrix}\right)} \ar[r]^{\left(\begin{smallmatrix}
			x_{2}  & g\\ 0 & 1
			\end{smallmatrix}\right)} & B_{2} \oplus B_{1}\ar[r]^{\left(\begin{smallmatrix}
			-y_2, y_1
			\end{smallmatrix}\right)}\ar[d]_{\left(\begin{smallmatrix}
			1,0
			\end{smallmatrix}\right)}  & C \ar@{-->}[r]^{-(\delta' \oplus 0)} &\\
		B_{2}\ar@{=}[r]\ar@{-->}[d]_{\delta''} & B_{2}\ar@{-->}[d]_{0}	 &   \\
		&&&}
	$$	
	such that the first  column is an $\mathbb{E}$-triangle, $\left(\begin{smallmatrix}
	-f\\x_1 \end{smallmatrix}\right)_*\delta=-(\delta'\oplus 0)$, $x_1$$_{*}\delta''=0$ and $(1,0)^{*}\delta''+(-y_{2},y_1)^{*}\delta=0$.
	We have
	$$ \delta''= ((1,0)\left(\begin{smallmatrix}
	1\\0 \end{smallmatrix}\right))^*\delta''=\left(\begin{smallmatrix}
	1\\0 \end{smallmatrix}\right)^*(1,0)^{*}\delta''=\left(\begin{smallmatrix}
	1\\0 \end{smallmatrix}\right)^*(y_{2},-y_1)^{*}\delta=((y_{2},-y_1)\left(\begin{smallmatrix}
	1\\0 \end{smallmatrix}\right))^{*}\delta=y_{2}^{*}\delta.$$
	It remains to show that $f_*\delta=\delta'$. In fact, by (ET3)$^{\text{op}}$ we have the following morphism of $\mathbb{E}$-triangles.
	$$\xymatrix{
		A_{2} \oplus B_{1}\ar@{-->}[d]_{\left(\begin{smallmatrix}
			-1, b
			\end{smallmatrix}\right)} \ar[r]^{\left(\begin{smallmatrix}
			x_{2}  & g\\ 0 & 1
			\end{smallmatrix}\right)} & B_{2} \oplus B_{1}\ar[r]^{\left(\begin{smallmatrix}
			-y_2, y_1
			\end{smallmatrix}\right)}\ar[d]_{\left(\begin{smallmatrix}
			-1,0
			\end{smallmatrix}\right)}  & C \ar@{-->}[r]^{-(\delta' \oplus 0)}\ar@{=}[d] &\\
		A_{2}\ar[r]^{x_2} & B_{2}\ar[r]^{y_2}	& C\ar@{-->}[r]^{\delta'} &  \\
	}$$ Note that $(-1,b)_*(-(\delta'\oplus 0))=\delta'$ and $(bx_1)_*\delta=0$ by Lemma \ref{lem 2.1}, we have
	$$f_*\delta=f_*\delta+(bx_1)_*\delta=(-1,b)_*\left(\begin{smallmatrix}
	-f\\x_1 \end{smallmatrix}\right) _*\delta=(-1,b)_*(-(\delta'\oplus 0))=\delta'.$$
	Therefore, (ET4-5) holds.
\end{proof}

\begin{thm}
	Let $(\mathscr{C},\mathbb{E},\mathfrak{s})$ be a pre-extriangulated category. If $\mathscr{C}$ satisfies  $(\mathrm{ET0})$  and  $(\mathrm{ET0})^{\textup{op}}$, then  $(\mathrm{ET4})$, $(\mathrm{ET4}$-$1)$, $(\mathrm{ET4}$-$2)$, $(\mathrm{ET4}$-$3)$, $(\mathrm{ET4}$-$4)$ and $(\mathrm{ET4}$-$5)$ are equivalent and self-dual.
\end{thm}

\begin{proof} Since $\mathscr{C}$ satisfies  $(\mathrm{ET0})$ and  $(\mathrm{ET0})^{\textup{op}}$, by Theorem \ref{thm2} and Theorem \ref{thm3}, $(\mathrm{ET4})$, $(\mathrm{ET4}$-$1)$, $(\mathrm{ET4}$-$2)$, $(\mathrm{ET4}$-$3)$ and $(\mathrm{ET4}$-$5)$ are equivalent and self-dual. In this case, if $\mathscr{C}$ satisfies (ET4), then $\mathscr{C}$ is an extriangulated category. Thus  $(\mathrm{ET4})$ implying (ET4-4) follows by Lemma \ref{lem2.3}. We complete the proof by Lemma \ref{lem3.1}.
\end{proof}


\begin{thebibliography}{99}
	
	\bibitem {HLN} M. Herschend, Y. Liu, H. Nakaoka, $n$-exangulated categories (I): Definitions and fundamental properties, J. Algebra, 2021 (570): 531-586.
	
	\bibitem {HZZ} J.S. Hu, D.D. Zhang, P.Y. Zhou, Proper classes and Gorensteinness in extriangulated categories, J. Algebra,
	2020 (551): 23-60.
	
	\bibitem {Hu} A. Hubery, Notes on the octahedral axiom, Available from the author's web page at
	http://www.math.uni-paderborn.de/~hubery/Octahedral.pdf.
	
	\bibitem{Kr}  H. Krause, Derived categories, resolutions, and Brown representability, Contemporary Mathematics, 2007, 436: 101.
	
	\bibitem {LN} Y. Liu, H. Nakaoka, Hearts of twin cotorsion pairs on extriangulated categories, Journal of Algebra, 2019, 528: 96-149.
	
	\bibitem {Nee}  A. Neeman,  Triangulated Categories, Ann. of Math. Stud., 148, Princeton
	University Press, Princeton, NJ, 2001.
	
	\bibitem {NP} H. Nakaoka, Y. Palu, Mutation via Hovey twin cotorsion pairs and model structures in extriangulated categories, Cahies De Topologie Et Geometrie Differntielle Categoriques, 2019, 60(2): 117-193.
	
	\bibitem {ZZ} Y.P. Zhou, B. Zhu, Triangulated quotient categories revisited, J. Algebra, 2018, (502): 196-232.
	
\end{thebibliography}
\end{document}